\documentclass[reqno,11pt]{amsart}
\usepackage{amssymb}
\usepackage{amsmath}
\usepackage{amsthm}
\usepackage{lipsum}

\newtheorem{theorem}{Theorem}[section]

\newtheorem{cor}[theorem]{Corollary}
\theoremstyle{definition}
\newtheorem{definition}[theorem]{Definition}
\newtheorem{example}[theorem]{Example}
\newtheorem{proposition}[theorem]{Proposition}

\theoremstyle{remark}
\newtheorem{remark}[theorem]{Remark}

\numberwithin{equation}{section}

    \newcommand{\st}{^{\textstyle {\ast }}}
   \newcommand{\ph}{\mbox{$\varphi$}}
    
    \renewcommand{\phi}{\varphi}
   \newcommand{\Ht}{\mbox{$H^{2}$}}
   
   \newcommand{\Hi}{\mbox{$H^{\infty}$}}
   \newcommand{\D}{\mbox{$\mathbb{D}$}}
   \newcommand{\C}{\mbox{$C_{\varphi}$}}
   \newcommand{\Cs}{\mbox{$C_{\varphi}\st$}}

   \newcommand{\W}{\mbox{$W_{\psi,\varphi}$}}
   \newcommand{\Ws}{\mbox{$W_{\psi,\varphi}^{\textstyle\ast}$}}
 	\newfont{\caps}{cmcsc9}  % for Authors
 	\newfont{\jour}{cmti9}  % for Journal Titles
\title{Normaloid Weighted Composition Operators on $H^2$}
\author{Derek Thompson}%
\address{Taylor University, Upland, IN, 46989}%
\email{theycallmedt@gmail.com}%
\thanks{The author was funded by the Taylor University Distinguished Lecturer award.}

\begin{document}

\maketitle

\begin{abstract}
    When \ph\ is an analytic self-map of the unit disk with Denjoy-Wolff point $a \in \D$, and $\rho(\W) = \psi(a)$, we give an exact characterization for when \W\ is normaloid. We also determine the spectral radius, essential spectral radius, and essential norm for a class of non-power-compact composition operators whose symbols have Denjoy-Wolff point in \D. When the Denjoy-Wolff point is on $\partial \D$, we give sufficient conditions for several new classes of normaloid weighted composition operators.
    
\end{abstract}

\section{Introduction} \label{intro}

In this paper, we are interested in weighted composition operators on the classical \textit{Hardy space} \Ht, the Hilbert space of analytic functions $\displaystyle f(z) = \sum_{n=0}^{\infty}a_n z^n$ on the open unit disk \D\ such that $$\|f\|^{2}=\sum_{n=0}^{\infty}|a_n|^{2}<\infty.$$ A \textit{composition operator} \C\ on \Ht\ is given by $\C f = f \circ \ph$. When \ph\ is an analytic self-map of \D, the operator \C\ is necessarily bounded. A \textit{Toeplitz operator} $T_\psi$ on $H^2$ is given by $T_\psi f = P(\psi f)$ where $P$ is the projection back to $H^2$. When $\psi \in \Hi$, the space of bounded analytic functions on \D, we simply have $T_\psi f = \psi f$, since $\psi f$ is guaranteed to be in \Ht, and all such Toeplitz operators are bounded. Throughout this paper, we will assume $\psi \in \Hi$.  We write $\W = T_\psi \C$ and call such an operator a \textit{weighted composition operator}. We are interested in when such operators are \textit{normaloid}. 

For a bounded operator $T$, we have the following definitions:

\begin{itemize}
    \item $\sigma(T)$ is the spectrum of $T$.
    \item $\rho(T)$ is the spectral radius of $T$.
    \item $\rho_{e}(T)$ is the essential spectral radius of $T$.
    \item $W(T)$ is the numerical range of $T$.
    \item $\| T \|_{e}$ is the essential norm of $T$.
    \item $r(T)$ is the numerical radius of $T$.
\end{itemize}

An operator $T$ is:

\begin{enumerate}
    \item \textit{self-adjoint} if $T = T^*$.
    \item \textit{normal} if $T^*T = TT^*$.
    \item \textit{hyponormal} if $T^*T \geq TT^*$.
    \item \textit{cohyponormal} if $T^*T \leq TT^*$.
    \item \textit{normaloid} if $\| T \| = \rho(T)$.
    \item \textit{convexoid} if the closure of $W(T)$ is the convex hull of $\sigma(T)$.
    \item \textit{spectraloid} if $\rho(T) = r(T) $.
    \item \textit{power-compact} if $T^{n}$ is compact for some positive integer $n$. 
\end{enumerate}

Here is a list of well-known facts which we will use repeatedly:

\begin{enumerate}
    \item We have the following sequences of implications: 
    $$ \textrm{self-adjoint} \Rightarrow \textrm{normal} \Rightarrow \textrm{(co)hyponormal} \Rightarrow \textrm{normaloid / convexoid} \Rightarrow \textrm{spectraloid} $$
    \item if $\psi \in \Hi$, then $\| T_{\psi} \| = \| \psi \|_{\infty}$. 
    \item For a bounded operator $T$, we have 
    $$ \| T \|_{e} = \inf \{ \| T - Q \| : Q \textrm{ is compact} \}$$ $$\rho_{e}(T) = \lim_{k \rightarrow \infty} (\| T^{k} \|_{e})^{1/k}$$ $$\rho(T) = \lim_{k \rightarrow \infty} (\| T^{k} \|)^{1/k}.$$
\end{enumerate}

Throughout this paper, we will focus on weighted composition operators where $\rho(\W) = |\psi(a)| \rho(\C)$, where $a$ is the Denjoy-Wolff point of \C. This is a large class, including every power-compact weighted composition operator \cite[Theorem 4.3]{Hammond}, and many weighted composition operators whose compositional symbol converges uniformly to its Denjoy-Wolff point \cite[Corollary 10]{derek}. Due to the norm inequality $\| \W \| \leq \| T_\psi \| \| \C \| = \|\psi\|_{\infty} \| \C \|$, we will often assume $|\psi(a)|=\|\psi\|_{\infty}$. It is unclear whether this is necessary, but we do have $\| \psi \|_{2}$ as a lower bound for $|\psi(a)|$ when $\rho(\C) = 1$.

\begin{proposition}\label{psi2} Suppose \ph\ is an analytic  self-map of \D\ with Denjoy-Wolff point $a$, $\psi \in \Hi$, \W\ is normaloid, and $\rho(\W) = |\psi(a)|$. Then $\| \psi \|_{2} \leq |\psi(a)|$. \end{proposition}    
\begin{proof} $$|\psi(a)| = \rho(\W) = \| \W \| \geq \| \W 1 \|_{2} = \|\psi \|_{2}. $$ \end{proof}

The organization of the rest of the paper is as follows. In Section 2, we consider the case when the Denjoy-Wolff point $a$ of \ph\ belongs to \D. In Section 3, we show that if $a$ belongs to $\partial \D$, the set of operators for which $\rho(\W) = |\psi(a)| \rho(\C)$ is non-trivial. For such operators, we discover new classes of normaloid weighted composition operators in Section 4. We end with further questions about normaloid weighted composition operators in Section 5.

\section{$a \in \D$}\label{intsec}

When the Denjoy-Wolff point $a$ of \ph\ is in \D, \C\ is rarely normaloid, as the next theorem shows.

\begin{theorem}\label{ph0} If the Denjoy-Wolff point of \ph\ is in \D, then \C\ is normaloid if and only if $\ph(0)=0$.\end{theorem}

\begin{proof} By \cite[Theorem 3.9]{Carlbook}, the spectral radius of \C\ is 1. By \cite[Corollary 3.7]{Carlbook}, we have $$ \left( \frac{1}{1-|\ph(0)|^2} \right)^{1/2} \leq \| \C \| \leq \left( \frac{1+|\ph(0)|}{1-|\ph(0)|} \right)^{1/2}$$ and we have $\| \C \|= 1$ if and only if $\ph(0)=0$. \end{proof}

Unsurprisingly, then, we show that if $\rho(\W) = |\psi(a)|\rho(\C)$, then \W\ is normaloid if and only if $\psi$ has a particular form. For the interior fixed point case, since $\rho(\C) = 1$, that assumption is really $\rho(\W) = |\psi(a)|$. Since this case also always has $|\psi(a)| \leq \rho(\W) \leq \| \W \|$, the next theorem focuses on when $|\psi(a)| = \| \W \|$. 

\begin{theorem}\label{interior} Suppose \ph\ is an analytic self-map of the disk with Denjoy-Wolff point $a \in \D$, and $\psi \in \Hi$. Then $|\psi(a)| = \| \W \|$ if and only if $\psi$ has the form $$\psi = \psi(a) \frac{K_{a}}{K_{a}\circ\varphi}.$$ \end{theorem}
\begin{proof}

First, assume the two values are equal. Suppose $\ph(0)=0$ so that $\| \W \| = |\psi(0)|$. However, $$|\psi(0)| = \| \W \| \geq \| \W 1 \|= \| \psi \| = \sqrt{\left|\psi(0)\right|^{2} + \left|\psi'(0)\right|^2+\left|\frac{\psi''(0)}{2}\right|^2 + \dots}$$ and we only have equality if every derivative of $\psi$ at $0$ is $0$, making $\psi$ constant (equal to $\psi(0)$), which trivially fits the required form, since $K_0 = 1$. 

Now, suppose $\ph(a) = a$, for some $a \in \D$ other than $0$. The weighted composition operator $W_{\zeta, \tau}$ where $\zeta = \sqrt{1-|a|^2} \frac{1}{1-\overline{a} z}, \tau = \frac{a- z}{1 - \overline{a} z}$ is unitary by \cite[Theorem 6]{Bourdon}. Note that $\tau$ switches $a$ and $0$ and is an involution and $W_{\zeta, \tau}$ is its own inverse. Therefore $W_{\zeta, \tau} \W W_{\zeta, \tau}$ is unitarily equivalent to \W, and it is again a weighted composition operator $W_{f,g}$, where $f = (\zeta) (\psi \circ \tau )(\zeta \circ \varphi \circ \tau)$ and $g = \tau \circ \varphi \circ \tau$. Since $g(0) = 0$, by the same logic as above, $f$ is a constant function, and the constant is $f(0) = \psi(a)$. Therefore, we have 

$$\psi \circ \tau = \frac{\psi(a)}{(\zeta) (\zeta \circ \varphi \circ \tau)}$$ 
and now composing both sides with $\tau$, and recalling $\tau \circ \tau = z$, we have

\begin{align}
\psi &= \frac{\psi(a)}{(\zeta \circ \tau)( \zeta \circ \varphi)}\label{eqn2} \\
    &=  \psi(a) \left(\frac{1}{1-\overline{a}z}\right) (1 - \overline{a}\varphi) \nonumber \\
    &= \psi(a) \frac{K_{a}}{K_{a}\circ\varphi}\label{eqn}.
\end{align}

For the other direction, suppose $\psi$ has the form given in equation (\ref{eqn}), and we will show that $|\psi(a)| = \| \W \|$. By the same logic as the other direction, \W\ is unitarily equivalent to $W_{\zeta, \tau} \W W_{\zeta, \tau}$, and again, this is a weighted composition operator of the form $W_{f,g}$, $f = (\zeta) (\psi \circ \tau )(\zeta \circ \varphi \circ \tau), g = \tau \circ \varphi \circ \tau$. Using the form for $\psi$ from Equation (\ref{eqn2}), we see that 

$$f = (\zeta) \left(\frac{\psi(a)}{(\zeta \circ \tau \circ \tau)(\zeta \circ \varphi \circ \tau)}\right)(\zeta \circ \varphi \circ \tau) = \psi(a).$$

Then $W_{f,g} = \psi(a)C_{g}$, and since $g(0)=0$, $\|C_{g} \| = 1$, so we have $\| \W \| = \| W_{f,g} \| = \| \psi(a) C_{g} \| = |\psi(a)| \| C_g \| = |\psi(a)|$. 
\end{proof}

From this, we have an immediate corollary about when \W\ is normaloid.

\begin{cor}\label{result} Suppose \ph\ is an analytic self-map of the disk with Denjoy-Wolff point $a \in \D$, $\psi \in \Hi$, and $\rho(\W) = |\psi(a)|$. Then \W\ is normaloid if and only if $\psi$ has the form $$\psi = \psi(a) \frac{K_{a}}{K_{a}\circ\varphi}.$$
\end{cor}
\begin{proof}
Note that we have $|\psi(a)| \leq \rho(\W) \leq \| \W \|$ since $a \in \D$. Then if \W\ is normaloid, we have $\| \W \| = \rho(\W) = |\psi(a)|$, where the second equality is by hypothesis. If instead we assume $\psi$ has the given form, by Theorem \ref{interior}, we have $|\psi(a)| = \| \W \|$, therefore $|\psi(a)| = \rho(\W) = \| \W \|$, so \W\ is normaloid.
\end{proof}

In the previous corollary, we assumed that $\rho(\W) = |\psi(a)| = |\psi(a)|\rho(\C)$. The next corollary shows this case includes all power-compact weighted composition operators with $\psi \in \Hi$, and Section 3 shows it includes several different classes of \W\ with Denjoy-Wolff point of \ph\ on $\partial \D$.  

\begin{cor} Suppose \ph\ is an analytic self-map of \D\ with Denjoy-Wolff point $a \in \D$, $\psi \in \Hi$, and \W\ is power-compact. Then \W\ is normaloid if and only if $\psi$ has the form $$\psi = \psi(a) \frac{K_{a}}{K_{a}\circ\varphi}.$$ \end{cor}

\begin{proof} By \cite[Proposition 4.3]{Hammond}, if \W\ is compact, then $\rho(\W) = |\psi(a)|$, so Theorem \ref{interior} applies. If \W\ is power-compact, we still have that the spectral radius is the absolute value of its largest eigenvalue, and $\Ws (K_{a}) = \overline{\psi(a)}K_{a}$, so $\rho(\W) = |\psi(a)|$.  \end{proof}

This form for $\psi$ is not unexpected, since the same form is required for \W\ to be normal when the fixed point of \ph\ belongs to \D\ \cite[Proposition 8]{Bourdon}, or even for \W\ to be cohyponormal (in \cite{coko}, they are shown to be equivalent when $a \in \D)$). However, while normality requires a much stricter characterization for \ph, here we show that this form for $\psi$ is sufficient for \W\ to be normaloid, while allowing for many different forms for \ph. 

At the end of \cite{derek}, the authors ask how often we have $\sigma(\W) = \psi(a)\sigma(\C)$. The work above shows that there are weighted composition operators with \ph\ having Denjoy-Wolff point $a$ where this is false.

\begin{example}\label{2-z} In \cite[Theorem 3.7]{derek2}, examples are given of weights $\psi \in \Hi$ such that \W\ is hyponormal when $\ph(z) = \frac{sz}{1-(1-s)z}, 0 < s < 1$. Every hyponormal operator is normaloid, but the weights are not as prescribed in Corollary \ref{result}. Therefore, it must be that $\rho(\W) > |\psi(0)|$. \end{example}

When \ph\ has Denjoy-Wolff point $a \in \D$, the primary differentiation of spectrum comes from whether or not \ph\ has a fixed point or a periodic point on $\partial \D$. Here, we obtain a partial result, that gives the spectral radius in Example \ref{2-z}. The proof of the following theorem is heavily borrowed from theorems about \C\ in \cite{Carlbook}.

\begin{remark}\label{remark1} Let $\ph_{n}$ denote the $n$th iterate of \ph, i.e. $\ph_{n} = \ph \circ \ph \dots \circ \ph$, $n$ times. By the discussion ahead of \cite[Theorem 7.36]{Carlbook}, when \ph\ has Denjoy-Wolff point in \D, is analytic in a neighborhood of the closed disk, and is not an inner function, there is an integer $n$ so that the set $S_n = \{ w: |w| = 1 \textrm{ and } |\ph_{n}(w)| = 1 \}$ is either empty or consists only of the finitely many fixed points of $\ph_n$ on the circle. The essential spectral radius of $\C$ is $$\rho_{e}(\C)=\max \{ \ph'_{n}(w)^{-1/2n} : w \in S_{n} \}. $$ If $S_{n}$ is empty, then \C\ is power-compact, which we have covered, so we will assume $S_{n}$ is nonempty. We will say that the chosen element $b$ of $S_{n}$ \textit{establishes} $\rho_{e}(\C)$. \end{remark}

\begin{theorem}\label{essspec} Suppose \ph, not an inner function, is an analytic self-map of \D\ which is univalent on \D\ and analytic in a neighborhood of $\overline{\D}$, with Denjoy-Wolff point $a \in \D$. Suppose $b \in \partial \D$ establishes $\rho_{e}(\C)$. Let $\psi \in \Hi$ be continuous at $b$ and let $|\psi(b)| = \| \psi \|_{\infty}$. Then 
\begin{enumerate}
    \item $\| \W \|_{e} = |\psi(b)| \| \C \|_{e}$,
    \item $\rho_{e}(\W) = |\psi(b)| \rho_{e}(\C)$, and
    \item  $\rho(\W) = \max \{|\psi(a)|, |\psi(b)|\rho_{e}(\C) \}$.
\end{enumerate}

\end{theorem}

\begin{proof}

By \cite[Theorem 7.31]{Carlbook}, we have that $$\rho_{e}(\C) = \lim_{k \rightarrow \infty} 
\left(\limsup_{|w| \rightarrow 1} \frac{\|K_{\ph_k(w)}\|}{\|K_w\|} \right)$$ and this happens in particular as $w$ approaches the element $b$ of $\partial \D$ that gives the maximum value in the definition in Remark \ref{remark1}. 

Now, we adapt the proof from \cite[Proposition 3.13.]{Carlbook} to show that $$\| \W \|_{e} \geq |\psi(b)| \| \C \|_{e}.$$

Let $w_j$ be a sequence in $\D$ tending to $\partial \D$. Then the normalized weight sequence $k_j = \frac{K_{w_j}}{\| K_{w_j} \|}$ tends to $0$ weakly as $j$ approaches infinity. If $Q$ is an arbitrary compact operator on \Ht, then $Q^{*}(k_j) \rightarrow 0$. 

\noindent Now, $\| \W \|_{e} = \inf \{ \| \W - Q\|: Q \textrm{ is compact} \}$, and for $Q$ compact,
\begin{align*}
\|\W - Q\| &\geq \limsup_{j \rightarrow \infty} \|(\W - Q)^{*}k_j\| \\&= \limsup_{j \rightarrow \infty} \|\Ws k_j\|\\ &=\limsup_{j \rightarrow \infty} |\psi(w_j)| \| \Cs k_j\|. 
\end{align*}
Since  $\| \C \|_{e} = \displaystyle \limsup_{j \rightarrow \infty} \| \Cs k_j\|$ is achieved by taking $w_j$ tending towards $b$, and likewise $\displaystyle \limsup_{j \rightarrow \infty} |\psi(w_j)| = |\psi(b)| = \|\psi \|_{\infty}$ is achieved by taking $w_j$ towards $b$, we have $\displaystyle \limsup_{j \rightarrow \infty} |\psi(w_j)| \| \Cs k_j\| = |\psi(b)| \| \C \|_{e}$ and $\| \W \|_{e} \geq |\psi(b)| \| \C \|_{e}$. 

For the other direction, note that since the compact operators are an ideal, $T_\psi Q$ is compact for any compact operator $Q$, and if $B \subseteq A$, then $\inf A \leq \inf B$. Then 

\begin{align*}
    \| \W\|_{e} &= \inf \{ \| \W - Q\|: Q \textrm{ is compact} \} \\
    &\leq \inf \{ \| \W - T_{\psi} Q \|: Q \textrm{ is compact} \} \\
    &= \inf \{ \| T_\psi (\C - Q) \| : Q \textrm{ is compact} \} \\
    &\leq \inf \{ \| T_{\psi} \| \| \C - Q \|: Q \textrm{ is compact} \} \\ 
    &= \inf \{ \| \psi \|_{\infty} \| \C - Q \|: Q \textrm{ is compact} \} \\ 
    &= \inf \{ |\psi(b)| \| \C - Q \|: Q \textrm{ is compact} \} \\ 
    &=|\psi(b)| \inf  \{ \| \C - Q \|: Q \textrm{ is compact} \}\\
    &=|\psi(b)| \| \C \|_{e}.
\end{align*}

\noindent Therefore we have $\| \W \|_{e} = |\psi(b)| \| \C \|_{e}$. 

Suppose momentarily that $b$ is a fixed point of \ph. Since \ph\ is analytic in a neighborhood of $\overline{\D}$ (i.e. continuous at $b$), the above gives the same result if $\psi$ is replaced by $\psi \circ \ph_{k}$ for any $k$.  Then,

\begin{align*}\rho_{e}(\W) &= \lim_{k \rightarrow \infty} \left(\| \W^{k}\|_{e}\right)^{1/k} \\ &=\lim_{k \rightarrow \infty} \left(\| T_{(\psi) (\psi \circ \varphi) (\psi \circ \varphi_2) \cdots (\psi \circ \varphi_{k-1})} C_{\varphi_{k}}\|_{e}\right)^{1/k}\\ &= \lim_{k \rightarrow \infty} \left( |\psi(b)|^{k} \| C_{\varphi_{k}}\|_{e}\right)^{1/k} \\&= |\psi(b)|\lim_{k \rightarrow \infty} \left( \| C_{\varphi_{k}}\|_{e}\right)^{1/k}\\ &= |\psi(b)|\rho_{e}(\C).
\end{align*}

\noindent Now, while $b$ may not be a fixed point of \ph, we know it is the fixed point of some $n$th iterate $\ph_{n}$. Furthermore, we know that $\rho_{e}(\W) =\displaystyle \lim_{k \rightarrow \infty} \left(\| \W^{k}\|_{e}\right)^{1/k}$ is a convergent sequence. Therefore, every subsequence converges to $\rho_{e}(\W)$. By taking the subsequence with indices $nk$ as $k \rightarrow \infty$, we see that $\rho_{e}(\W) = |\psi(b)|\rho_{e}(\C)$. 

The complement in $\sigma(\W)$ of the essential spectrum consists of eigenvalues of finite multiplicity. By \cite[Proposition 4.3]{Hammond}, any eigenvalue of \W\ must be of the form $$\{ 0, \psi(a), \psi(a)\ph'(a), \psi(a)(\ph'(a))^2, \psi(a)(\ph'(a))^3, \dots \}$$

\noindent and the largest of those values in magnitude is $\psi(a)$. This value is necessarily in $\sigma(\W)$, since $\Ws(K_{a}) = \overline{\psi(a)}K_{a}$. 

Therefore, $\rho(\W) = \max \{|\psi(a)|, |\psi(b)|\rho_{e}(\C) \}$. 

\end{proof}

\begin{cor}\label{intnormaloid} Suppose \ph, not an inner function, is an analytic self-map of \D\ which is univalent on \D\ and analytic in a neighborhood of $\overline{\D}$, with Denjoy-Wolff point $a \in D$. Suppose $b \in \partial \D$ establishes $\rho_{e}(\C)$. Let $\psi \in \Hi$ be continuous at $b$ and let $|\psi(b)| = \| \psi \|_{\infty}$.

Suppose further that \W\ is normaloid. Then, either $\| \W \| = \rho(\W) = |\psi(a)|$ and $\psi = \psi(a) \frac{K_{a}}{K_{a}\circ\varphi}$, or $ \| \W \| = \rho(\W) = \rho_{e}(\W) = |\psi(b)|\rho_{e}(\C)$.
\end{cor}

\begin{proof} By Theorem \ref{essspec}, $\rho(\W) = \max \{|\psi(a)|, |\psi(b)|\rho_{e}(\C) \}$. If $\rho(\W) = |\psi(a)|$, by Theorem \ref{interior}, $\psi$ has the form $\psi = \psi(a) \frac{K_{a}}{K_{a}\circ\varphi}$. Otherwise, $\rho(\W) = |\psi(b)|\rho_{e}(\C)$. 

\end{proof}

\begin{example} Let \W\ be the hyponormal operator given by $\psi(z) = \frac{2e^{z}}{(2-z)}, \ph(z) = \frac{z}{2-z}$ \cite[Example 3.8]{derek2}. Then $\rho(\W) = |\psi(1)|\ph'(1)^{-1/2} = \sqrt{2}e$. \end{example}

While neither Theorem \ref{essspec} nor Corollary \ref{intnormaloid}  gives sufficient conditions for \W\ to be normaloid, Theorem \ref{essspec} does accomplish something else. An operator $T$ is said to be \textit{essentially normaloid} if $\rho_{e}(T) = \| T \|_{e}$. 

\begin{cor} Suppose \ph, not an inner function, is an analytic self-map of \D\ which is univalent on \D\ and analytic in a neighborhood of $\overline{\D}$, with Denjoy-Wolff point $a \in D$. Let $b$ be a fixed point of \ph\ on $\partial D$ such $b$ that establishes $\rho_{e}(\C)$. Let $\psi \in \Hi$ be continuous at $b$ and let $|\psi(b)| = \| \psi \|_{\infty}$. Then \W\ is essentially normaloid if and only if \C\ is essentially normaloid. \end{cor}

\begin{proof} This is an immediate consequence of $(1)$ and $(2)$ in Theorem \ref{essspec}. \end{proof}

\begin{example} Let $\ph(z) = \frac{z}{2-z}, \psi = e^{-z} $. Then $|\psi(0)| = 1$ and $|\psi(1)|\ph'(1)^{-1/2} = \frac{\sqrt{2}}{2e} < 1$. Since $\psi$ is not of the form $|\psi(0)|\frac{K_{0}}{K_{0} \circ \ph} \equiv 1$, \W\ is not normaloid. However, since \C\ is essentially normaloid \cite[Theorem 7.31, 7.36]{Carlbook}, therefore so is \W. \end{example}

\section{Uniformly Convergent Iteration (UCI)}

We now turn our attention to when \ph\ has Denjoy-Wolff point $a \in \partial \D$. We wish to continue to assume that $\rho(\W) = |\psi(a)|\rho(\C)$. Our goal in this section, before determining when such operators are normaloid, is to show that this class is non-trivial. To do so, we will put restrictions on how the iterates of \ph\ converge to the Denjoy-Wolff point. This definition is from \cite{derek}, where this hypothesis is used to determine the spectrum of weighted composition operators in this setting. 

The Denjoy-Wolff Theorem \cite[Theorem 2.51]{Carlbook} states that all analytic self-maps of \D\ other than elliptic automorphisms have a point in $\overline{\D}$ that they converge to under iteration on compact subsets of \D. Here, we ask the convergence to be stronger. 

\begin{definition}[Uniformly Convergent Iteration] We say \ph\ is UCI if \ph\ is an analytic self-map of \D\ and the iterates of \ph\ converge uniformly to the Denjoy-Wolff point uniformly on \textit{all} of \D, rather than compact subsets of \D. \end{definition} 

If \ph\ is UCI and the Denjoy-Wolff point $a$ of \ph\ belongs to \D, then \W\ is power-compact \cite[Corollary 2]{derek}, so we have already covered that scenario in Section \ref{intsec} without requiring this additional hypothesis. 

Analytic self-maps of \D\ that exhibit UCI while having Denjoy-Wolff point $a$ on $\partial \D$ are a non-trivial set, and include maps whose derivative at the Denjoy-Wolff point are both less than $1$ (e.g. $\ph(z) = (z+1)/2$) and equal to $1$ (e.g. $\ph(z) = 1/(2-z)$) \cite[Example 5]{derek}. A simple sufficient condition for UCI when $\ph'(a) < 1$ is that $\phi_{N}(\overline{\D}) \subseteq \D \cup \{ a \}$ for some $N$ \cite[Theorem 4]{derek}. This includes, then, any linear-fractional map with Denjoy-Wolff point on the boundary and $\ph'(a) < 1$.

The main reason to now introduce this definition is the following theorem, proved in \cite{derek}.

\begin{theorem}\label{uci} Suppose \ph\ is UCI with Denjoy-Wolff point $a \in \partial \D$, $\psi \in \Hi$ is continuous at $a$, and $\psi(a) \neq 0$.Then:
\begin{enumerate}
    \item $\overline{\sigma_{p}(\psi(a)\C)}\subseteq\overline{\sigma_{ap}(T_{\psi}\C)}\subseteq\sigma(T_{\psi}\C)\subseteq\sigma(\psi(a)\C),$
\item  If $\overline{\sigma_{p}(\C)}=\sigma(\C)$,
then $\sigma(T_{\psi}\C)=\sigma(\psi(a)\C)$, 
\item If $\ph'(a) < 1$, then $\sigma(T_{\psi}\C)=\sigma(\psi(a)\C)$ \textbf{and} $\sigma_{p}(\W)=\sigma_{p}(\psi(a)\C)$. 
\end{enumerate}

\end{theorem}

We have an immediate corollary regarding the spectral radius.

\begin{cor}  Suppose \ph\ is UCI with Denjoy-Wolff point $a \in \partial \D$, $\psi \in \Hi$ is continuous at $a$, and $\psi(a) \neq 0$. Then $\rho(\W) = |\psi(a)|\rho(\C).$ \end{cor}

\begin{proof} If $\ph'(a) < 1$, then $(3)$ of the Theorem \ref{uci} makes this clear. If $\ph'(a) = 1$, note that $1 \in \sigma_{p}(\C)$, so $\psi(a) \in \sigma_{p}(\psi(a)\C)$, and therefore by $(1)$ of Theorem \ref{uci}, $\rho(\W) \geq |\psi(a)|$. Again by $(1)$, we also have $\sigma(\W) \subseteq \sigma(\psi(a)\C)$, so $\rho(\W) \leq \rho(\psi(a)\C) = |\psi(a)|\rho(\C) = |\psi(a)|$, since $\rho(\C) = 1$ \cite[Theorem 3.9]{Carlbook}. Therefore $\rho(\W) = |\psi(a)| = |\psi(a)|\rho(\C).$   \end{proof}

\section{$a \in \partial \D$}

Here we follow the same path as Section \ref{intsec}. We will continue to assume that we have $\rho(\W) = |\psi(a)|\rho(\C)$, and will seek to determine conditions for which $\| \W \|$ is the same. We will also obtain a few corollaries for when \ph\ is explicitly UCI. 

\begin{theorem}\label{bdy} Suppose \ph\ is an analytic self-map of \D\ with Denjoy-Wolff point $a \in \partial \D$, $\psi \in \Hi$, and $\psi$ is  continuous at the Denjoy-Wolff point $a$ of \ph, with $\| \psi \|_{\infty} = |\psi(a)|$. Furthermore, assume $\rho(\W) = |\psi(a)|\rho(\C).$ If \C\ is normaloid,  then \W\ is normaloid. \end{theorem}

\begin{proof} Note that 
\begin{align*}
    \| \W \| &\leq \| T_{\psi} \| \| \C \| \\ &= \|\psi\|_{\infty} \| \C \| \\ &= |\psi(a)| \| \C \| \\ &= |\psi(a)|\rho(\C) \\ &= \rho(\W) \leq \| \W \|. 
\end{align*}

\end{proof}

\begin{example} If $\psi = e^{z}$ and $\ph = (z+1)/2$, then $\| \psi \|_{\infty} = e = \psi(1)$. Since \C\ is cohyponormal and therefore normaloid \cite[Theorem 8.7]{Carlbook}, \W\ is normaloid and also convexoid. \end{example}

While examples generated by Theorem \ref{bdy} are reasonable to come by when $\ph'(a) < 1$, they are actually impossible to come by when $\ph'(a) = 1$. 

\begin{theorem} Suppose \ph\ is an analytic self-map of \D\ with Denjoy-Wolff point $a \in \partial \D$ and $\ph'(a) = 1$. Then \C\ is not normaloid. \end{theorem}

\begin{proof} The proof is analogous to Theorem \ref{ph0}. The spectral radius for \C\ is $\ph'(a)^{-1/2}$ when $a \in \partial \D$ \cite[Theorem 3.9]{Carlbook}, so here we have $\rho(\C) = 1$. Since $\ph(0) \neq 0$, we know $\| \C \| > 1$, therefore \C\ is not normaloid. \end{proof}

However, there \textit{are} known weighted composition operators where $\ph'(a) = 1$ and \W\ is normaloid - even self-adjoint \cite{cgk}. We make a minor adjustment to Theorem \ref{bdy} to generate new examples of normaloid weighted composition operators in this setting. 

\begin{cor} Suppose \ph\ is an analytic self-map of \D\ with Denjoy-Wolff point $a$, $\psi \in \Hi$ is continuous at $a$, and $f \in \Hi$ is also $f$ is continuous at $a$, with $\| f \|_{\infty} = |f(a)|$. If \W\ is normaloid and $\rho(W_{f\psi, \ph}) = |f(a)|\rho(\W)$, then $W_{f\psi, \ph}$ is normaloid. \end{cor}
\begin{proof} 

The proof is identical to Theorem \ref{bdy}, with a mere adjustment of symbols:
\begin{align*}
    \| W_{f\psi, \varphi} \| &\leq \| T_{f} \| \| \W \|\\ &= \| f \|_{\infty} \| \W \| \\ &= |f(a)| \| \W \| \\ &= |f(a)|\rho(\W) \\ &= \rho(W_{f\psi, \varphi}) \\ &\leq \| W_{f\psi, \varphi} \|. 
\end{align*}

\end{proof}

\begin{example} Suppose $\psi(z) = \frac{1}{2-z}, f(z) = e^{z}$. Then $W_{\psi, \psi}$ is self-adjoint and therefore normaloid by \cite[Theorem 6]{cgk}. Since $\psi$ is UCI \cite[Example 5]{derek},we have $\rho(W_{f\psi, \psi}) = |f(1)||\psi(1)| = |f(1)|\rho(W_{\psi, \psi})$. Since $|f(1)| = e = \| f \|_{\infty}$, we have that $W_{f\psi, \psi}$ is normaloid. \end{example}

We end this section with a few extra facts for when \ph\ is UCI and $\ph'(a) < 1$. 

\begin{theorem}\label{specvex} Suppose \ph\ is UCI, the Denjoy-Wolff point $a$ of \ph\ is on $\partial \D$, and $\ph'(a)<1$. Then \W\ is convexoid if and only if \W\ is spectraloid. \end{theorem}

\begin{proof} Every convexoid operator is spectraloid. In the other direction, assume \W\ is spectraloid, so that $\rho(\W) = r(\W)$. Note that by $(3)$ of Theorem \ref{uci}, the spectrum of \W\ is a closed disk centered at the origin, completely filling in the set $\{ \lambda \in \mathbb{C} : |\lambda| \leq \rho(\W) \}$. Since $\rho(\W) = r(\W)$, this set is necessarily also the closure of the numerical range. Therefore, \W\ is convexoid.  \end{proof}

\begin{cor} Suppose \ph\ is UCI, the Denjoy-Wolff point $a$ of \ph\ is on $\partial \D$, and $\ph'(a)<1$. If \W\ is normaloid, then \W\ is convexoid. \end{cor}

\begin{proof} Every normaloid operator is spectraloid, so \W\ is spectraloid. By Theorem \ref{specvex}, if \W\ is spectraloid, it is also convexoid. \end{proof}

\section{Further Questions}

Here we summarize the questions raised by the work of this paper. 

\begin{enumerate}
    \item If \W\ is normaloid and $\rho(\W) = |\psi(a)| \rho(\C)$, is it necessary that $|\psi(a)| = \|\psi\|_{\infty}$?
    \item Can the many hypotheses of Theorem \ref{essspec} be weakened, to identify $\rho_{e}$ in the general setting when \ph\ has interior Denjoy-Wolff point and \C\ is not power-compact?
    \item Can we then characterize all normaloid weighted composition operators where \ph\ has Denjoy-Wolff point in \D? 
    \item What are the necessary conditions for \W\ to be normaloid when the Denjoy-Wolff point of \ph\ is on $\partial \D$? 
    \item Ultimately, can we get an exact characterization of when \W\ is normaloid? 
\end{enumerate}

\section*{Acknowledgements}

The author would like to thank the Taylor University Distinguished Lecturer program for funding this research, and the reviewer for his extremely helpful comments. 

\footnotesize

\bibliographystyle{amsplain}

\begin{thebibliography}{99}




\bibitem{Bourdon} P. S. Bourdon, S. K. Narayan, Normal weighted composition operators on the Hardy
space $H^{2}(\mathbb{D})$, {\it J. Math. Anal. Appl.} {\bf 367} (2010), 278-286.



\bibitem{cgk} C. C. Cowen, G. Gunatillake, and E. Ko, Hermitian weighted composition operators and Bergman extremal functions. {\it Complex Anal. Oper. Theory,} {\bf 7(1)} (2013), 69-99.

\bibitem{Carlbook} C. C. Cowen and B.D. MacCluer, \textit{Composition Operators on Spaces of Analytic Functions}, CRC Press, Boca Raton, 1995. MR \textbf{97i}:47056.
 \bibitem{coko} C. C. Cowen, S. Jung, and E. Ko, Normal and cohyponormal weighted composition operators on $H^{2}$, {\it Operator Theory: Advances and Applications} {\bf  240} (2014), 69-85.
\bibitem{derek} C. C. Cowen, E. Ko, D. Thompson and F. Tian, Spectra of some weighted composition operators on $H^{2}$, {\it Acta Sci. Math. (Szeged),} {\bf 82} (2016), 221-234.

\bibitem{derek2} M. Fatehi, M. Haji Shaabani, Thompson, D., Quasinormal and Hyponormal Weighted Composition Operators on $H^2$ and $A^2_{\alpha }$ with Linear Fractional Compositional Symbol, {\it Complex Analysis and Operator Theory}. (2017). doi:10.1007/s11785-017-0683-3


\bibitem{Hammond} C. Hammond, \textit{On the Norm of a Composition Operator}, Thesis, University of Virginia, 2003. 






\end{thebibliography}

\end{document}